\documentclass[12pt, twoside, leqno]{article}

\usepackage{amsmath,amsthm}
\usepackage{amssymb}
\usepackage{enumerate}
\usepackage[T1]{fontenc}

\DeclareMathAlphabet{\mathpzc}{OT1}{pzc}{m}{it}

\newtheorem{thm}{Theorem}

\newtheorem{lem}[thm]{Lemma}

\newtheorem{propp}[thm]{Property}

\newtheorem{mainthm}[thm]{Main Theorem}

\theoremstyle{definition}
\newtheorem{defin}[thm]{Definition}
\newtheorem{rem}[thm]{Remark}

\makeindex

\newcommand{\CC}{\mathbb{C}}

\newcommand{\NN}{\mathbb{N}}

\begin{document}

\baselineskip=17pt

\title{Weighted generalization of the Ramadanov theorem and further considerations.}

\author{by
\sc Zbigniew Pasternak-Winiarski \rm (Warsaw) \\
and
\sc Pawe{\l} M. W\'ojcicki \rm (Warsaw)
}

\date{}
\maketitle
\renewcommand{\thefootnote}{}
\renewcommand{\thefootnote}{\arabic{footnote}}
\setcounter{footnote}{0}

    \begin{abstract}
    We study the limit behavior of weighted Bergman kernels on a sequence of domains in a complex space $\CC^N$, and show that under some conditions on domains and weights, weighed Bergman kernels converge uniformly on compact sets. Then we give a weighted generalization of the theorem given in \cite[p. 38]{Skwarczy\'nski}, highlighting some special property of the domains, on which the weighted Bergman kernels converge uniformly. Moreover we will show that convergence of weighted Bergman kernels implies this property, which will give a characterization of the domains, for which the inverse of Ramadanov's theorem holds.

    \end{abstract}

    \section{Introduction}

    The Bergman kernel (see for instance \cite{Bergman, Jarnicki; Pflug, Krantz1, Krantz2, Shabat, Skwarczy\'nski; Mazur} ) has become a very important tool in geometric function theory, both in one and several complex variables. It turned out that not only classical Bergman kernel, but also weighted one can be useful. Let $D\subset\CC^N$ be a bounded domain. For example (see \cite{Englis}), if we denote by $\Pi : L^2(D)\rightarrow L^2_H(D)$ (the Bergman projection), we may define for any $\psi\in L^{\infty}(D)$, the Toeplitz operator $T_{\psi}$ as a (bounded linear) operator on $L^2_H(D)$ by $T_{\psi}f:=\Pi(\psi f)$. In particular, for $\psi>0$ on $D$ we have that $T_{\psi}$ is positive definite (so injective), so there exists an inverse $T_{\psi}^{-1}$. Taking positive continuous weight function $\mu\in L^{\infty}(D)$, $T_{\mu}$ extends to a bounded operator from $L^2_H(D,\mu)$ into $L^2_H(D)$, and $K_{D,\,\mu}(\cdot,x)=T_{\mu}^{-1}K_{D}(\cdot,x)$, where $K_{D,\,\mu}(\cdot,x)$ denotes the weighted Bergman kernel (associated to weighted Bergman space $L^2_H({D,\,\mu})$) at $x\in D$. Another practical application of weighted Bergman kernel may bo found in quantum theory (see \cite{Englis1} and  \cite{Odzijewicz}, and \cite{Pasternak-Winiarski})- we may consider a K\"ahler manifold $\Omega$ as a classical phase space of a physical system (many leading quantized classical systems have such a phase space). The Hilbert space $H$ of quantum states of such a system consists of the holomorphic sections of some Hermitian line bundle $E$ over $\Omega$, which belong to $L^2(\Omega,\,\mu)$ (for the Liouville measure $\mu$ on $\Omega$). One of the most interesting and important objects of this model is the reproducing kernel $K$ of $H$ (that is the kernel $K_{\Omega,\,\mu}$ ). This kernel makes the quantization of classical states possible as follows : one can assign to any classical state $z\in\Omega$ the quantum state $$v_z:=[K(\cdot,z)/||K(\cdot,z)||]\in H.$$ Using this embedding one can calculate the transition probability amplitude from one point to another : $$a(z,w):=|<v_z | v_w>|,\quad z,w\in\Omega.$$ Then the calculation of the Feynman path integral for such a system is equivalent to finding the reproducing kernel $K$ (that is $K_{\Omega,\,\mu}$). \\
      But in general, it is difficult to say anything about the unweighted (regular) or weighted kernel of a given domain. One of the classic results for unweighted Bergman kernels is Ramadanov's theorem (see \cite{Ramadanov}) :
    \begin{thm}[Ramadanov]
    Let $D_1\Subset D_2\Subset D_3\ldots$ be an increasing sequence of domains and set $D:=\bigcup_j D_j$. Then, $K_{D_j}\rightarrow K_D$ uniformly on compact subsets of $D\times D.$
    \end{thm}
    It is very natural to ask whether the similar theorem for weighted Bergman kernels is true. Let's recall the Forelli-Rudin construction (see \cite{Forelli; Rudin} and \cite{Ligocka}) :
    If $\mu$ is a continuous weight on $D$ and $\Omega$ denotes the Hartogs domain $$\Omega=\{(z,w)\in D\times \CC^n : ||w||^{2n}<\mu(x)\}$$ in $\CC^{N+n} $, then $$K_{D,\mu}(z,p)=\frac{\pi^n}{n!}K_{\Omega}((z,0),(p,0))$$ (that is the weighted Bergman kernel $K_{D,\mu}(z,p)$ of $D$ is the restriction of the unweighted Bergman kernel $K_{\Omega}((z,w),(p,s))$ of $\Omega$ to the hyperplane $w=s=0$). Thus using Ramadanov's th. for the kernels $K_{\Omega_j}((z,0),(p,0))$ we can derive (under some conditions on weights - monotonicity for instance) the weighted analog of this theorem. And in fact, we may find some versions of this theorem in
    ({{\rm\cite[Prop. 3.17; Th. 3.18]{Jacobson}}} for instance), but considered weights are in the special form, as a moduli of holomorphic functions or $C^2$ functions, or as a product of one of those with the given weight $\psi$. Additionally, unweighted generalization of Ramadanov th. was given in \cite{Krantz} (in weighted case we can't proceed similarly, since we would have to strictly restrict weights of the kernels (we would need further assumptions for instance that $\mu_{\Omega_j}\circ\Theta_j=\mu_\Omega$ for any diffeomorphism $\Theta_j : \Omega\rightarrow\Omega_j$)). We can easily see that continuity of weight $\mu$ in the Forelli-Rudin construction provides basically that $\Omega$ is an open set. In this paper we will derive a weighted version of Ramadanov's theorem for so called "admissible weights" $\mu$ (we don't require $\mu$ to be continuous) without using Forelli - Rudin construction. It is very natural to consider such kind of weights, just by their definition (see below). We will prove the inverse of this theorem as well (see also \cite[p. 37]{Skwarczy\'nski} for an unweighted situation). In the second part of the paper we will show that density of holomorphic functions on a considered domain is very related to the convergence of the weighted Bergman kernels. In fact, we will get an equivalence in the unweighted case. This will provide with characterization of the domains, for which the inverse of Ramadanov's theorem holds. We shall start from the definitions and basic facts used in this paper.

    \section{Definitions and notations}

   Let $D\subset\CC^N$ be a domain, and let $W(D)$ be the set
   of weights on $D$, i.e., $W(D)$ is the set of all Lebesque measurable, real - valued,
   positive functions on $D$ (we consider two weights as equivalent if they are equal almost
   everywhere with respect to the Lebesque measure on $D$).
   If $\mu\in W(D)$, we denote by $L^2(D,\mu)$ the space of all Lebesque measurable, complex-valued,
   $\mu$-square integrable functions on $D$, equipped with the norm $||\cdot||_{D, \mu}:=||\cdot||_{\mu}$
   given by the scalar product
   \begin{align*}
   <f|g>_{\mu}:=\int_Df(z)\overline{g(z)}\mu(z)dV,\quad f, g\in L^2(D,\mu).
   \end{align*}
   The space $L^2_H(D,\mu)=H(D)\cap L^2(D,\mu)$ is called the \textbf{weighted Bergman space}, where $H(D)$ stands the space of all holomorphic functions on the domain $D$. For any $z\in D$ we define the evaluation functional $E_z$ on $L^2_H(D,\mu)$ by the formula
   \begin{align*}
   E_zf:=f(z),\quad f\in L^2_H(D,\mu).
   \end{align*}
   Let us recall the definition [Def. 2.1] of admissible weight given in \cite{Pasternak-Winiarski1}.
   \begin{defin}[Admissible weight]
   A weight $\mu\in W(D)$ is called an{ \textit{admi\-ssible weight} }, an a-weight for short, if $L^2_H(D,\mu)$
   is a closed subspace of $L^2(D,\mu)$ and for any $z\in D$ the evaluation functional $E_z$ is
   continuous on $L^2_H(D,\mu)$. The set of all a-weights on $D$ will be denoted by $AW(D)$.
   \end{defin}
   The definition of admissible weight provides us basically with existence and uniqueness of related Bergman kernel and completeness of the space $L^2_H(D,\mu)$.
   In \cite{Pasternak-Winiarski} the concept of a-weight was introduced, and in \cite{Pasternak-Winiarski1} several theorems concerning admissible
   weights are given. An illustrative one is :
   \begin{thm}\cite[Cor. 3.1]{Pasternak-Winiarski1}
   Let $\mu\in W(D)$. If the function $\mu^{-a}$ is locally integrable on $D$ for some $a>0$ then $\mu\in AW(D)$.
   \end{thm}
   Now, let 's fix a point $t\in D$ and minimize the norm $||f||_{\mu}$ in the class $E_t=\{f\in L^2_H(D,\mu); f(t)=1\}$.
   It can be proved in a similar way as in the classical case, that if $\mu$ is an admissible weight then there exists exactly one function minimizing the norm. Let us denote it by $\phi_{\mu}(z,t)$.
   \textbf{Weigted Bergman kernel function $K_{D,\,\mu}$} is defined as follows :
   \begin{align*}
   K_{D,\,\mu}(z,t)=\frac{\phi_{\mu}(z,t)}{||\phi_{\mu}||_{\mu}^2}.
   \end{align*}
\section{Variations on the Ramadanov theorem and domain dependence.}
In this section we study the limit behavior of weighted Bergman kernels for admissible weights. Moreover we give a weighted characterization of the Bergman kernel (see also \cite[p. 36]{Skwarczy\'nski}) by means of which we prove kind of converse of Ramadanov theorem. We show that density of holomorphic functions is very related to the convergence of weighted Bergman kernels, and in the case of $\mu_n\equiv 1$ we even have an equivalence (see also \cite{Skwarczy\'nski; Iwi\'nski}).
 \subsection{Weighted generalization of the Ramadanov theorem }
   \begin{mainthm}[Weighted generalization of the Ramadanov theorem]
   Let $\{D_i\}_{i=1}^{\infty}$ be a sequence of domains in $\CC^N$
   and set $D:=\bigcup_jD_j$. Let $\mu\in AW(D),\mu_k\in AW(D_k)$ (extend $\mu_k$ by $\mu$ on $D$).
   Assume moreover that
   \begin{enumerate}
     \item[a)]For any $n\in\NN$ there is $N=N(n)$ s.t. $D_n\subset D_m$ and $\mu_n(z)\leq\mu_m(z)\leq\mu(z)$ for $m\geq N(n),\,\,z\in D_n$.
     \item[b)] $\displaystyle\mu_k\xrightarrow[k\to\infty]{}\mu$ pointwise a.e. on $D$.
   \end{enumerate}
   Then
   \begin{align*}
   \lim_{k\to\infty}K_{D_k, \mu_k}=K_{D, \mu}
   \end{align*}
   locally uniformly on $D\times D$.
   \end{mainthm}
  The first step in the proof is to show the monotonicity property for the weighted kernels.
   Then we should check that the limit of the sequence of weighted kernels of the domains $D_n$,
   if exists, is equal to $K_{D,\mu}$.
    \begin{lem}[Monotonicity property]
    For any $n\in\NN,\,\,t\in D_n$ the inequality $K_{D_n, \mu_n}(t,t)\geq K_{D_m, \mu_m}(t,t)$ holds for $m\geq N(n)$.
    \end{lem}
     \begin{proof}
     Let us fix $n\in\NN,\,\,t\in D_n$. Let $m\geq N(n)$.
     The inequality in the statement of the lemma is true if $K_{D_m,\,\mu_m}(t,t)=0$. Then suppose that
     $K_{D_m,\,\mu_m}(t,t)>0$.
     In the proof we will use the simple remark that
     \begin{align*}
     \frac{1}{K_{D_n,\,\mu_n}(t,t)}= \int_{D_n}\left|\frac{K_{D_n,\,\mu_n}(s,t)}{K_{D_n,\,\mu_n}(t,t)}\right|^2\mu_n(s)dV
     \end{align*}
      since
     $K_{D_n,\,\mu_n}(t,t)>0$ and
     \begin{align*}
     K_{D_n,\,\mu_n}(t,t)=\displaystyle\int_{D_n}\overline{K_{D_n,\,\mu_n}(z,t)}
     K_{D_n,\,\mu_n}(z,t)\mu_n(z)dV
     \end{align*}
     by the reproducing property (\cite{Pasternak-Winiarski}) for $f(\cdot)=K_{D_n,\,\mu_n}(\cdot,t)$. Moreover the term
     $\displaystyle\frac{K_{D_n,\,\mu_n}(\cdot,t)}{K_{D_n,\,\mu_n}(t,t)}$
     is the only element in the class $\{f\in L^2_H(D_n,\mu_n), f(t)=1 \}$ with the minimal norm.
     Thus for $m\geq N(n)$ we have

     \begin{align*}
  \frac{1}{K_{D_n,\,\mu_n}(t,t)}
  &\leq \int_{D_n}\left|\frac{K_{D_m,\,\mu_m}(s,t)}{K_{D_m,\,\mu_m}(t,t)}\right|^2\mu_n(s)dV\\[12pt]
  &\leq \int_{D_n}\left|\frac{K_{D_m,\,\mu_m}(s,t)}{K_{D_m,\,\mu_m}(t,t)}\right|^2\mu_m(s)dV \\[12pt]
  &\leq \int_{D_m}\left|\frac{K_{D_m,\,\mu_m}(s,t)}{K_{D_m,\,\mu_m}(t,t)}\right|^2\mu_m(s)dV=
     \frac{1}{K_{D_m,\,\mu_m}(t,t)}.
    \end{align*}

    \end{proof}
   \begin{rem}
     One can show similarly that $K_{D_n,\,\mu_n}(t,t)\geq K_{D,\,\mu}(t,t)$ for \\$n\in\NN$.
   \end{rem}
   \begin{lem}[Uniqueness of the limit]
    If $\displaystyle\lim_{n\to\infty}K_{D_n,\,\mu_n}=k$ locally uniformly on $D\times D$, then
    $k=K_{D,\,\mu}$.
    \end{lem}
     \begin{proof}
     Since the sequence $(K_{D_n,\,\mu_n})$ converges locally uniformly on $D\times D$ and any function
     $K_{D_n,\,\mu_n}$ is continuous we obtain that $k$ is continuous on $D\times D$.
     Let's recall that
     \begin{align}\label{E:1}
  \int_{D_m}\overline{K_{D_m,\,\mu_m}(z,t)}K_{D_m,\,\mu_m}(z,t)
     \mu_m(z)dV=K_{D_m,\,\mu_m}(t,t).
    \end{align}
     Fix a compact set $E\subset D$, and $t\in E$. For $m$ large enough $E\subset D_m$ and $t\in D_m$.
     By Fatou's lemma
     \begin{align*}
  \int_E|k(z,t)|^2\mu(z)dV
  &\leq \liminf_{m\to\infty}\int_{E}|K_{D_m,\,\mu_m}(z,t)|^2\mu_m(z)dV \\[12pt]
  &\leq  \liminf_{m\to\infty}\int_{D_m}|K_{D_m,\,\mu_m}(z,t)|^2\mu_m(z)dV\\[12pt]
  &= \liminf_{m\to\infty}K_{D_m,\,\mu_m}(t,t)=k(t,t).
    \end{align*}
     Since $E$ is an arbitrary compact set,
     \begin{align}\label{E:1}
   &\int_D|k(z,t)|^2\mu(z)dV\leq k(t,t).
    \end{align}
     By Weierstrass theorem $k(\cdot,t)\in H(D)$, so $k(\cdot,t)\in L^2_H(D,\mu)$.
     \newline By Lemma $5$ we get
     \begin{align*}
     K_{D_n,\,\mu_n}(t,t)\geq K_{D,\,\mu}(t,t)
     \end{align*}
     for $n=1, 2,\ldots,\quad t\in D$. In the limit $n\to\infty$ we obtain
     \begin{align*}
     k(t,t)\geq K_{D,\,\mu}(t,t).
     \end{align*}
     It suffices to show that $k(z,t)=K_{D,\,\mu}(z,t)$.
     We should consider two cases :
     \\
     \\
     $1.$ $K_{D,\,\mu}(t,t)=0$, for some $t\in D.$
     \\
     \\
     Then for $z\in D$,
     $K_{D,\,\mu}(z,t)=0$ since $\displaystyle K_{D,\,\mu}(t,t)=\int_D|K_{D,\,\mu}(z,t)|^2\mu(z)dV,$
     and $K_{D,\,\mu}$ is continuous with respect to $z.$ Thus for any $f\in L^2_H(D,\mu)$
     \begin{align*}
     f(t)=\int_Df(w)K_{D,\,\mu}(t,w)\mu(w)dV=0
     \end{align*}
     and we have that $k(t,t)=0$, since $f(\cdot):=k(\cdot,t)\in L^2_H(D,\,\mu).$ \\But $\displaystyle\int_D|k(z,t)|^2\mu(z)dV\leq k(t,t),$
     so $k(z,t)=0$, for $z\in D.$
     \\
     \\
     $2.$ $K_{D,\,\mu}(t,t)>0$, for some $t\in D$.
     \\
     \\
     Then $k(t,t)>0,$ since $k(t,t)\geq K_{D,\,\mu}(t,t)>0.$ We will use once more the well known fact, that in the set $\{f\in L^2_H(D,\mu), f(t)=1 \}$ (for some fixed $t\in D$) function
     $\displaystyle\frac{K_{D,\,\mu}(\cdot,t)}{K_{D,\,\mu}(t,t)}$ is the only minimal element.
     It is easy to see, that $\displaystyle\frac{k(\cdot,t)}{k(t,t)}$ belongs to this set
     (since $k(\cdot,t)\in L^2_H(D,\mu)$) and moreover by (2) $||k(\cdot,t)||_{\mu}\leq\sqrt{k(t,t)}.$ \\Thus
     \begin{align*}
  \left|\left|\frac{k(\cdot,t)}{k(t,t)}\right|\right|_{\mu}
  &\leq \frac{\sqrt{k(t,t)}}{k(t,t)}=\frac{1}{\sqrt{k(t,t)}}\leq \frac{1}{\sqrt{K_{D,\,\mu}(t,t)}}\\[12pt]
  &= \left|\left|\frac{K_{D,\,\mu}(\cdot,t)}{K_{D,\,\mu}(t,t)}\right|\right|_{\mu}.
    \end{align*}

     By the minimality property of $\displaystyle\frac{K_{D,\,\mu}(\cdot,t)}{K_{D,\,\mu}(t,t)}$
     we get from the above, that :

     \begin{align*}
  \left|\left|\frac{k(\cdot,t)}{k(t,t)}\right|\right|_{\mu}=
     \frac{1}{\sqrt{k(t,t)}}= \frac{1}{\sqrt{K_{D,\,\mu}(t,t)}}=
     \left|\left|\frac{K_{D,\,\mu}(\cdot,t)}{K_{D,\,\mu}(t,t)}
     \right|\right|_{\mu}.
    \end{align*}
     So $k(t,t)=K_{D,\,\mu}(t,t)$, and $k(z,t)=K_{D,\,\mu}(z,t)$ for $z, t\in D$.
     \end{proof}
     \begin{proof}[Proof of the main theorem]
     We will show that for $n\in\NN$ the sequence $\{K_{D_m,\,\mu_m}\}_{m\geq N(n)}$ is locally bounded on $D_n\times D_n$.
     \\
     Using well known version of Schwarz inequality for reproducing kernels and Lemma $5$
     we obtain for any $z, t\in D_n$.
     \begin{align*}
  |K_{D_m,\,\mu_m}(z,t)|
  &\leq \sqrt{K_{D_m,\,\mu_m}(z,z)}\sqrt{K_{D_m,\,\mu_m}(t,t)}\\[12pt]
  &\leq \sqrt{K_{D_n,\,\mu_n}(z,z)}\sqrt{K_{D_n,\,\mu_n}(t,t)},\quad m\geq N(n).
    \end{align*}
     The term in the right hand side of the estimation above is locally bounded on $D_n\times D_n.$
     By Montel's property, any subsequence of $\{K_{D_m,\,\mu_m}\}$ has a subsequence convergent locally
     uniformly on $D\times D$. By Lemma $7$ the limit does not depend on a subsequence and is
     identically equal to $K_{D,\,\mu}$. Thus
     \begin{align*}
     \lim_{m\to\infty}K_{D_m,\,\mu_m}(z,t)=K_{D,\,\mu}(z,t)
     \end{align*}
     locally uniformly on $D\times D$.
     \end{proof}
  \begin{rem}
  Look that the case of increasing sequence of domains is a subcase of the Main Theorem $4$ (see \cite{Skwarczy\'nski; Mazur} for the very interesting considerations, and unweighted kind of Lemma $7$).
  \end{rem}
    \subsection{Characterization of the weighted Bergman kernel and further remarks on "decreasing-like"
    sequence of domains}

    In \cite[p. 36]{Skwarczy\'nski} some characterization lemma for unweighted Bergman kernels is given.
    One can easily conclude similar one for weighted Bergman kernels, as the following Lemma $9$ shows.
    The proof is attached only for the convenience of the reader.

    \begin{lem}
    Denote by $S_{\mu,\,t}\subset L^2_H(D,\mu)$ the set of all functions $f$ such that $f(t)\geq 0$
    and $||f||_{\mu}\leq\sqrt{f(t)}$, where $t\in D$ is fixed. Then the weighted Bergman
    function $\varphi_{\mu,\,t}(\cdot):=K_{D,\,\mu}(\cdot,t)$ is uniquelly characterized by the proper\-ties :
    \begin{enumerate}
    \item[(i)] $\varphi_{\mu,\,t}\in S_{\mu,\,t}$
    \item[(ii)] if $f\in S_{\mu,\,t}$ and $f(t)\geq \varphi_{\mu,\,t}(t)$, then $f\equiv\varphi_{\mu,\,t}$.
    \end{enumerate}
    \end{lem}
    \begin{proof}
    One can easily see, that there exists at most one element $\varphi_{\mu,\,t}\in L^2_H(D,\mu)$
    which satisfies (i) and (ii) (if $\varphi_1, \varphi_2$ satisfies (i) and (ii), then both $\varphi_1(t)$ and $\varphi_2(t)$ are nonnegative, and either $\varphi_1(t)\geq\varphi_2(t)$ and then $\varphi_1\equiv\varphi_2$ or $\varphi_2(t)\geq\varphi_1(t)$ and then $\varphi_2\equiv\varphi_1$).
    We shall show $\varphi_{\mu,\,t}(\cdot)=K_{D,\,\mu}(\cdot,t)$ has both properties.
    \newline We have
    \begin{align*}
    \varphi_{\mu,\,t}(t)=K_{D,\,\mu}(t,t)\geq 0
    \end{align*}
    and
    \begin{align*}
    ||K_{D,\,\mu}(\cdot,t)||_{\mu}^2=K_{D,\,\mu}(t,t)
    \end{align*}
    (see (1)).
    Now let $f\in S_{\mu,\,t}$. If $f(t)=0$, then $\varphi_{\mu,\,t}(t)=0$. Hence $||f||_\mu =
    ||\varphi_{\mu,\,t}||_\mu=0$, so $f\equiv 0\equiv \varphi_{\mu,\,t}$.
    \\
    Assume now $f(t)>0$. By the definition of weighted Bergman kernel function $\displaystyle\frac
    {\varphi_{\mu,\,t}(\cdot)}{\varphi_{\mu,\,t}(t)}$ is uniquely characterized as an element in the set
    $\{h\in L^2_H(D,\mu), \newline h(t)=1\}$ with the minimal norm. But $\displaystyle\frac{f(\cdot)}{f(t)}$
    belongs to this set as well, moreover
    \begin{align*}
    \Bigg\|\frac{f(\cdot)}{f(t)}\Bigg\|_\mu=
    \frac{||f||_\mu}{\sqrt{f(t)}\sqrt{f(t)}}
    \leq\frac{1}{\sqrt{f(t)}}\leq\frac{1}{\sqrt{\varphi_{\mu,\,t}(t)}}=
    \Bigg\|\frac{\varphi_{\mu,\,t}(\cdot)}{\varphi_{\mu,\,t}(t)}\Bigg\|_\mu.
    \end{align*}
    Thus (by minimality)
    \begin{align*}
    \frac{1}{\sqrt{f(t)}}=\frac{1}{\sqrt{\varphi_{\mu,\,t}(t)}}
    \end{align*}
    and for any $z\in D$
    \begin{align*}
    \frac{f(z)}{f(t)}=\frac{\varphi_{\mu,\,t}(z)}{\varphi_{\mu,\,t}(t)}.
    \end{align*}
     So $f\equiv\varphi_{\mu,\,t}$.
    \end{proof}
    Let's assume that $D={\rm{int}}(\overline{D})$ to exclude slit domains from our considerations (a disc with
    one radius removed for instance) and consider "decreasing - like" version of the Ramadanov th. Let us recall the definition of "approximation from outside" given in \cite[Def. V.6; p. 38]{Skwarczy\'nski}.
  \begin{defin}
   We say that a sequence of domains $\{D_n\}_{n=1}^{\infty}$ approximates $D$ from outside if $D\subset D_n$ for all $n$ and for each open $G$ such that $\overline{D}\subset G$ the inclusion $D\subset D_m\Subset G$ holds for all sufficiently large $m$.
  \end{defin}
   \begin{mainthm}
   Let $\{D_n\}_{n=1}^{\infty}$ be a sequence of domains in $\CC^N$ which approximates $D$ from outside and $\mu\in AW(D),\mu_k\in AW(D_k)$ (extend $\mu_k$ by $\mu_n$ on $D_n$ for $k\geq n$, and $\mu$ by $\mu_n$ on $D_n$). Assume moreover that
   \begin{enumerate}
     \item[a)]$\mu(z)\leq\mu_m(z)$\quad for \quad $m\in\NN,\,\,z\in D$.
     \item[b)]$\displaystyle\mu_k\xrightarrow[k\to\infty]{}\mu$ pointwise a.e. on $D$.
   \end{enumerate}
   Then $\{K_{D_m,\,\mu_m}\}_{m=1}^{\infty}$ converges to $K_{D,\,\mu}$ locally uniformly
      on $D\times D$ iff for any fixed $t\in D$
      \begin{align*}
      \lim_{m\to\infty}K_{D_m,\,\mu_m}(t,t)=K_{D,\mu}(t,t).
      \end{align*}
   \end{mainthm}
  \begin{proof}
   We shall only make sure, that the converse implication is true, since the necessity is obvious.
   Let $F\subset D$ be a compact set. Then there is a constant $M=M(F)$ such that $\displaystyle\max_{z\in F}|K_{D,\,\mu}(z,z)|\leq M$.
   By Schwarz inequality

   \begin{align*}
  |K_{D_m,\,\mu_m}(z,t)|
  &\leq \sqrt{K_{D_m,\,\mu_m}(z,z)}\sqrt{K_{D_m,\,\mu_m}(t,t)}\\[12pt]
  &\leq \sqrt{K_{D,\,\mu}(z,z)}\sqrt{K_{D,\,\mu}(t,t)}\leq M.
    \end{align*}
   for any $z,\,t\in F$. Thus $\{K_{D_m,\,\mu_m}\}_{m=1}^{\infty}$ is a Montel family on $D\times D$. It suffices
   to show that every convergent subsequence of this family converges to $K_{D,\,\mu}$.
   With no loss of generality let us consider $\{K_{D_m,\,\mu_m}\}$ itself and assume
   that it does converge to some $k$. For $t\in D$, by Fatou's lemma

   \begin{align*}
  \int_F|k(z,t)|^2\mu(z)dV
  &\leq \liminf_{m\to\infty}\int_F|K_{D_m,\,\mu_m}(z,t)|^2\mu_m(z)dV\\[12pt]
  &\leq \liminf_{m\to\infty}\int_{D_m}|K_{D_m,\,\mu_m}(z,t)|^2\mu_m(z)dV \\[12pt]
  &= \liminf_{m\to\infty}K_{D_m,\,\mu_m}(t,t)=K_{D,\,\mu}(t,t)=k(t,t).
    \end{align*}
   Since $F\subset D$ is an arbitrary compact set,
   \begin{align*}
   ||k(\cdot,t)||_{\mu}^2\leq k(t,t)=K_{D,\,\mu}(t,t)<\infty.
   \end{align*}
   Thus taking in Lemma $9$ $f(\cdot)=k(\cdot,t)$ we obtain $K_{D,\,\mu}(z,t)=k(z,t)$ for any $z,\,t\in D$.
   \end{proof}
   \begin{rem}
   Look that decreasing sequence of domains satisfies assumptions of the Main Theorem $11$.
   This theorem for classical Bergman kernels and decreasing case of domains could be found in \cite[p. 37]{Skwarczy\'nski}. Main Theorem $4$ could be proved in the same fashion using Lemma $9$ (look in \cite{W\'ojcicki}).
   \end{rem}

\subsection{Domain dependence}
  
  In this paragraph, among others, we will give a generalization of \cite[p. 38]{Skwarczy\'nski} for weighted Bergman kernels. Moreover we will show that the converse of this theorem holds as well. We shall start with notation used in this paragraph. The first thing is to extend the weights outside its natural domain (the domain may intersect). Let $\{D_n\}_{n=1}^{\infty}$ be an approximating sequence for $D$. Let $\mu_n\in AW(D_n)$ for $n\in\NN$. Then for $k\geq 2$ we define :

    \begin{displaymath}
    \widetilde{\mu_k(z)} = \left\{ \begin{array}{ll}
    \mu_k(z) & \textrm{ for $z\in D_k$ }\\
    \widetilde{\mu_{k-1}(z)} & \textrm{ for $z\in D_1\cup\ldots\cup D_{k-1}\setminus D_k$ }
    \end{array} \right.
    \end{displaymath}
    See that $\widetilde{\mu_k(z)}$ is well defined on $D_1\cup\ldots\cup D_k$.
    For example

    \begin{displaymath}
    \widetilde{\mu_3(z)} = \left\{ \begin{array}{ll}
    \mu_3(z) & \textrm{ for $z\in D_3$ }\\
    \widetilde{\mu_2(z)} & \textrm{ for $z\in D_1\cup D_{2}\setminus D_3$ }
    \end{array} \right.
    = \left\{ \begin{array}{ll}
    \mu_3(z) & \textrm{ for $z\in D_3$ }\\
    \mu_2(z) & \textrm{ for $z\in D_{2}\setminus D_3$ }\\
    \mu_1(z) & \textrm{ for $z\in(D_1\setminus D_{2})\setminus D_3$ }
    \end{array} \right.
    \end{displaymath}
    We will now define the extension of $\mu(z)$ outside $D$.
    Since $\{D_n\}_{n=1}^{\infty}$ is an approximating sequence for $D$ then for large $s\in\NN$ we have $D\subset D_s\Subset D_1$.
    We define
     \begin{displaymath}
    \widetilde{\mu(z)} = \left\{ \begin{array}{ll}
    \mu_1(z) & \textrm{ for $z\in D_1$ }\\
    \mu_2(z) & \textrm{ for $z\in D_2\setminus D_1$ }\\
    \mu_3(z) & \textrm{ for $z\in D_3\setminus(D_1\cup D_2)$ }\\
    \ldots \\
    \mu_{s-1}(z)&\textrm{ for $z\in D_{s-1}\setminus(D_1\cup D_2\ldots\cup D_{s-2})$ }
    \end{array} \right.
    \end{displaymath}

     Let $E\subset\CC^N$ be Lebesque measurable, $\mu$ be a-weight on $E$ (we set that $\mu\in AW(E)$ if for some open set $W\subset\CC^N$ s.t. $E\subset W$ there is $\nu\in AW(W)$ s.t. $\nu_{|E}=\mu$) and $L^2(E,\mu)$ be the Hilbert space of all complex-valued functions which are square $\mu-$ integrable on a set $E$ and holomorphic in the interior of $E$. Let's moreover $H(E,\mu)$ be the subset of $L^2(E,\mu)$ consisting of all functions possessing holomorphic extension to an open neighborhood of $E$. We will need the following :
  
  \begin{propp}
  $H(E,\mu)$ is dense in $L^2(E,\mu)$.
  \end{propp}
  \begin{mainthm}
  Suppose that $\overline{D}$ has Property $13$, and
  a sequence $D_m$ approximates $D$ from outside.
  Let $\mu\in AW(\overline{D}),\mu_k\in AW(D_k)$ (extend $\mu_k$ and $\mu_n$ as mentioned above). 
  Assume moreover, that for some $p\in\NN$
   \begin{enumerate}
     \item[a)]$\mu(z)\leq\mu_m(z)\leq\mu_{p}(z)$\,\,for\,\,$p\leq m$ and $z\in D_{p}$
     \item[b)]$\mu_{p}\in L^1(D_{p})$
     \item[c)]$\displaystyle\mu_k\xrightarrow[k\to\infty]{}\mu$ pointwise a.e. in $D_{p}$
   \end{enumerate}
  Then
  \begin{align*}
  \lim_{m\to\infty}K_{D_m,\mu_m}=K_{D,\mu}
  \end{align*}
   locally uniformly on $D\times D$.
  \end{mainthm}
  
  \begin{proof}
  Let $t\in D$ and $f\in L^2_H(D,\mu)$ be fixed. We can extend $f$ by $0$ on $\partial D$, to provide
  $f\in L^2(\overline{D},\mu)$. Consider any $h\in H(\overline{D},\mu)$. Then for $m$ large enough, $h\in L^2_H(D_m,\mu_m)$ (because of $b)$).
  We have
   \begin{align*}
  |h(t)|
  &= \left|\int_{D_m}h(z)K_{D_m,\,\mu_m}(z,t)\mu_m(z)dV\right|\\[12pt]
  &= \left|\int_{D_m}h(z)\mu_m(z)^{1/2}K_{D_m,\,\mu_m}(z,t)\mu_m(z)^{1/2}dV\right| \\[12pt]
  &\leq ||h||_{\mu_m}\left(\int_{D_m}|K_{D_m,\mu_m}(z,t)|^2\mu_m(z)dV\right)^{1/2}\\[12pt]
  &= ||h||_{\mu_m} K_{D_m,\mu_m}(t,t)^{1/2}.
    \end{align*}
  In the limit $m\to\infty$ we get (by Domin. Conv. Th.) $|h(t)|\leq k(t,t)^{1/2}||h||_{E,\mu}$\,,
  \newline where $\displaystyle k(t,t)=\lim_{m\to\infty}
  K_{D_m,\mu_m}(t,t)$. By density Property $13$ there is a sequence $\{h_m\}$ of functions in $H(\overline{D},\mu)$ such that $h_m\xrightarrow{L^2(\overline{D},\mu)}f$ (thus locally uniformly on $D$). So
  \begin{align*}
  |f(t)|\leq k(t,t)^{1/2}||f||_{\overline{D},\mu}= k(t,t)^{1/2}||f||_{D,\mu}.
  \end{align*}
  So for $f(t)=K_{D,\mu}(t,t)$ we have
  \begin{align*}
  |K_{D,\mu}(t,t)|=K_{D,\mu}(t,t)\leq k(t,t)^{1/2}||K_{D,\mu}||_{\mu}=k(t,t)^{1/2}K_{D,\mu}(t,t)^{1/2}.
  \end{align*}
  Thus $K_{D,\mu}(t,t)\leq k(t,t)$. Obviously $K_{D_m,\mu_m}(t,t)\leq K_{D,\mu}(t,t)$. In the limit
  $m\to\infty$ we get $k(t,t)\leq K_{D,\mu}(t,t)$. Therefore
  \begin{align*}
  K_{D,\mu}(t,t)=k(t,t)=\lim_{m\to\infty}K_{D_m,\mu_m}(t,t).
  \end{align*}
  The hypothesis follows from the Main Theorem $11$.
  \end{proof}
 What is interesting, it turns out that some kind of the converse of Main Theorem $14$ holds as well, namely :
 \begin{mainthm}
  Let $D\subset\CC^N$ be a domain and let $\mu$ be a weight on the closure $\overline{D}$ of $D$ in $\CC^N$ s.t. the measure of the boundary $\partial D$ given by $\mu$ is equal to $0$ and $\mu_{|D}\in AW(D)$. Suppose that for some sequence $\{D_n\}$ approximating $D$ from outside, and some sequence of admissible weights $\{\mu_n\}$
  (where $\mu_n\in AW(D_n)$)
   \begin{align*}
   \lim_{n\to\infty}{K_{D_n, \mu_n}}_{|D}=K_{D, \mu}
   \end{align*}
    holds locally uniformly on $D\times D$; for any $t\in D, K_{D_n,\mu_n}(\cdot,t)\in L^2_H(D,\mu)$ and
  \begin{align*}
  \lim_{n\to\infty}||K_{D_n,\mu_n}(\cdot,t)||^2_{\mu}=||K_{D,\mu}(\cdot,t)||^2_{\mu}=K_{D,\mu}(t,t).
  \end{align*}Then Property $13$ holds.
 \end{mainthm}
 \begin{proof}
 For any $t\in D$ we have
  \begin{align*}
  & ||{K_{D_n, \mu_n}}_{|D}(\cdot,t)-K_{D,\mu}(\cdot,t)||^2_{\mu} \\[12pt]
  &= \int_D(K_{D_n,\mu_n}(z,t)-K_{D,\mu}(z,t))\overline{(K_{D_n,\mu_n}(z,t)-K_{D,\mu}(z,t))}\mu(z)dV\\[12pt]
  &= \int_D|K_{D_n,\mu_n}(z,t)|^2\mu(z)dV - \int_D K_{D,\mu}(t,z)K_{D_n,\mu_n}(z,t)\mu(z)dV\\[12pt]
  &- \overline{\int_D K_{D,\mu}(t,z)K_{D_n,\mu_n}(z,t)\mu(z)dV}+\int_D|K_{D,\mu}(z,t)|^2\mu(z)dV \\[12pt]
  &= ||K_{D_n,\mu_n}(\cdot,t)||^2_{\mu}-K_{D_n,\mu_n}(t,t)-\overline{K_{D_n,\mu_n}(t,t)}+K_{D,\mu}(t,t)\\[12pt]
  &= ||K_{D_n,\mu_n}(\cdot,t)||^2_{\mu}-2K_{D_n,\mu_n}(t,t)+K_{D,\mu}(t,t).
  \end{align*}
 By assumptions
  \begin{align*}
   \lim_{n\to\infty}||{K_{D_n,\mu_n}}_{|D}(\cdot,t)-K_{D,\mu}(\cdot,t)||^2_{\mu}=0.
  \end{align*}
 which means that the closure in $L^2-$ norm
   \begin{align*}
   cl\{K_{D,\,\mu}(\cdot,t),t\in D\}\subset cl\{K_{D_n,\,\mu_n}(\cdot,t), t\in D, n\in\NN \}\subset L^2_H(\overline{D},\mu).
  \end{align*}
 On the other hand, by reproducing property (\cite{Pasternak-Winiarski})
  \begin{align*}
   cl \{K_{D,\,\mu}(\cdot,t),t\in D\}=L^2_H(D,\mu)=L^2_H(\overline{D},\mu).
  \end{align*}
 Taking into account that
  \begin{align*}
    \{K_{D_n,\mu_n}(\cdot,t),t\in D, n\in\NN \}\subset H(\overline{D},\mu)
  \end{align*}
 we obtain desired result.
 \end{proof}
 \begin{rem}
 Look also in \cite{Skwarczy\'nski; Iwi\'nski} for some considerations concerning unweighted, decreasing case of Main Thm. $15$ and very interesting remarks. Look that taking for any $n,\,\,\mu_n\equiv 1$ we get in fact that Property $13$ and hypothesis of the Main Theorem $14$ are equivalent, which gives us a description of the domains, for which "decresing-like" version of Ramadanov theorem holds. Moreover, using Main Theorem $4$ we can prove a weighted version of counterexample to the Lu Qi-Keng conjecture given in \cite{Boas}.
 \end{rem}

\section*{Acknowledgements}
We are very grateful to Prof. Harold Boas for his very helpful advices on the Main Theorem $4$ and other suggestions. We would like to express our thanks to Prof. Maciej Skwarczy\'nski $(\dagger)$ for his passion and a great book for students, and his contribution to the theory of Bergman kernels. We would like to thank our families for supporting us through their presence.

\bigskip

{\noindent
Zbigniew Pasternak-Winiarski\\
Faculty of Mathematics and Information Science\\
Warsaw University of Technology\\
Koszykowa 75,  00-662 Warsaw, Poland\\
E-mail: \,z.pasternak-winiarski@mini.pw.edu.pl}

\medskip

{\noindent
Pawe{\l} M. W\'ojcicki\\
Faculty of Mathematics and Information Science\\
Warsaw University of Technology\\
Koszykowa 75,  00-662 Warsaw, Poland\\
E-mail: \,p.wojcicki@mini.pw.edu.pl}

\bigskip



\rightline{}

\end{document}